\documentclass[11pt, reqno]{amsart}

\makeatletter
\def\theequation{\@arabic\c@equation}
\usepackage{amsmath,amsthm,amscd,amssymb,latexsym,upref}
\usepackage{versions}
\usepackage{enumerate}
\usepackage{graphicx}
\usepackage{amsmath, amssymb, amsfonts, amsthm}
\usepackage{mathrsfs}
\usepackage{hyperref}
\usepackage{microtype}

\usepackage{newtxtext}
\usepackage{newtxmath}

\newtheorem{definition}{Definition}[section]
\newtheorem{remark}[definition]{Remark}
\newtheorem{proposition}[definition]{Proposition}
\newtheorem{theorem}[definition]{Theorem}
\newtheorem{lemma}[definition]{Lemma}
\newtheorem{corollary}[definition]{Corollary}
\newtheorem{example}[definition]{Example}

\numberwithin{equation}{section}

\begin{document}
\title{A Metric Framework for Approximate Transitivity, Mixing, and Hypercyclicity}

\author{Hadi Obaid Alshammari}
\address{Department of Mathematics, College of Science, Jouf University, Sakaka, P.O. Box 2014, Saudi Arabia}
\email{hahammari@ju.edu.sa}

\author{Otmane Benchiheb}
\address{Department of Mathematics, Faculty of Sciences, Chouaib Doukkali University, El Jadida, Morocco}
\email{otmane.benchiheb@gmail.com}

\author{Dimitrios Chiotis}
\address{School of Physical and Chemical Sciences, Queen Mary University of London, United Kingdom}
\email{d.chiotis@qmul.ac.uk}



\begin{abstract}
	We study metric versions of transitivity, mixing, and hypercyclicity for
	continuous maps, based on intersections of the form
	\(
	f^{n}(U)\cap B_{\delta}(V)\neq\varnothing.
	\)
	We introduce $\delta$-topological transitivity, $\delta$-topological mixing,
	and a uniform-from-below version of $\delta$-mixing, and prove
	\(
	\mathrm{UFB\mbox{-}}\delta\text{-TM}
	\;\Rightarrow\;
	\delta\text{-TM}
	\;\Rightarrow\;
	\delta\text{-TT}.
	\)
	In the linear setting of separable F-spaces, we formulate a
	$\delta$-Hypercyclicity Criterion, prove that it implies
	$\delta$-hypercyclicity, and show that the classical Hypercyclicity Criterion
	implies the $\delta$-criterion for every $\delta>0$. We further show that
	this criterion yields eventual $\delta$-mixing along the underlying sequence.
	Finally, we discuss weighted backward shifts, derive sufficient conditions for
	$\delta$-topological mixing, and show that $\lambda B$ satisfies the
	$\delta$-Hypercyclicity Criterion for every $\delta>0$.
\end{abstract}
	
	\maketitle
	\sloppy
	\fussy

	\section{Introduction}
	
	The study of discrete-time dynamical systems, both linear and nonlinear, is a central
	topic in modern analysis. In the nonlinear setting, one considers a continuous map
	$f\colon M\to M$ on a metric space $(M,d)$, and the dynamics is generated by its
	iterates $(f^n)_{n\ge0}$. In the linear setting, $M$ is typically a Banach or a
	Fr\'echet space $X$, and $f$ is replaced by a continuous linear operator
	$T\colon X\to X$. Although the evolution is linear, the orbit
	$\{T^n x : n\ge0\}$ of a single vector may still exhibit very rich behaviour.
	This places linear dynamics at the crossroads of functional analysis, topological
	dynamics and ergodic theory; see, for instance,
	\cite{BayartMatheron2009,GEP2011}.
	
	A fundamental notion in both frameworks is \emph{hypercyclicity}. A continuous map
	$f\colon M\to M$ is called \emph{hypercyclic} if there exists $x\in M$ such that
	the forward orbit
	\[
	\operatorname{Orb}(x,f) := \{f^n(x) : n\ge0\}
	\]
	is dense in $M$. Such an $x$ is a \emph{hypercyclic point} for $f$.
	On separable metric spaces, hypercyclicity is closely related to topological
	transitivity, going back to Birkhoff \cite{Birkhoff1929}. In the linear context,
	the first explicit example was provided by Rolewicz \cite{Rolewicz1969}, who proved
	that $\lambda B$ is hypercyclic on $\ell^p$ whenever $|\lambda|>1$, where $B$
	denotes the backward shift. Weighted shifts have since become one of the basic model
	classes in linear dynamics, and the Hypercyclicity Criterion is now one of the main
	tools of the subject; see \cite{GEP2011} for a systematic account.
	
	Classical hypercyclicity, topological transitivity and topological mixing are
	inherently \emph{exact}: they require genuine intersections of iterates of open
	sets, or exact density of an orbit. In many metric situations, however, it is
	natural to allow a fixed spatial tolerance. This motivates the study of
	\emph{approximate} or \emph{metric} versions of these notions.
	
	An early step in this direction is due to Feldman \cite{Feldman2002}, who introduced
	$d$-dense orbits in linear dynamics. Motivated by this perspective, the aim of the
	present paper is to study quantitative analogues of transitivity, mixing and
	hypercyclicity on metric spaces and on separable F-spaces, and to clarify some of
	the relations between these approximate notions and their classical counterparts.
	
	\medskip
	
	\noindent\textbf{Metric framework.}
	Let $(M,d)$ be a metric space, let $f\colon M\to M$ be continuous, and fix a
	tolerance $\delta>0$. We consider approximate versions of topological
	transitivity and mixing based on intersections of the form
	$f^n(U)\cap B_\delta(V)\neq\varnothing$, where $U,V\subset M$ are nonempty open
	sets and $B_\delta(V)$ is the open $\delta$-neighbourhood of $V$. In this
	framework we study:
	\begin{itemize}
\item $\delta$-topological transitivity ($\delta$-TT),
\item $\delta$-topological mixing ($\delta$-TM),
\item uniform-from-below $\delta$-topological mixing (UFB-$\delta$-TM).
\end{itemize}
These properties form a natural hierarchy, and one of the basic results established
in the paper is the implication chain
\[
\mathrm{UFB\mbox{-}}\delta\text{-TM}
\;\Longrightarrow\;
\delta\text{-TM}
\;\Longrightarrow\;
\delta\text{-TT}.
\]
We also present examples showing that these notions are genuinely metric in nature,
that $\delta$-topological mixing is strictly weaker than classical topological
mixing, and that some converse implications fail in general.

\medskip

\noindent\textbf{Linear framework and $\delta$-hypercyclicity.}
In the linear setting, we work with continuous linear operators on separable
F-spaces endowed with a translation-invariant metric compatible with the topology.
We formulate a \emph{$\delta$-Hypercyclicity Criterion}, in the spirit of the
classical Hypercyclicity Criterion, and prove that in this setting it implies
$\delta$-hypercyclicity. We also show that the classical Hypercyclicity Criterion
implies the $\delta$-criterion for every $\delta>0$. In addition, we show that
the $\delta$-Hypercyclicity Criterion yields eventual $\delta$-mixing along the
underlying sequence.

\medskip

\noindent\textbf{Weighted shifts.}
We then apply the general framework to weighted backward shifts. We recall the
classical mixing criteria for unilateral and bilateral weighted shifts, derive from
them corresponding sufficient conditions for $\delta$-topological mixing, and
include the classical example of the weighted shift $\lambda B$, which satisfies
the $\delta$-Hypercyclicity Criterion for every $\delta>0$.

\medskip

\noindent\textbf{Organization of the paper.}
Section~2 introduces and studies approximate transitivity and mixing in metric
spaces. Section~3 is devoted to the $\delta$-Hypercyclicity Criterion and its
relation to the classical Hypercyclicity Criterion. Section~4 concerns weighted
shifts and related examples.

Throughout the paper, $(M,d)$ denotes a metric space and $f\colon M\to M$ a
continuous map. For $A\subset M$ and $r>0$ we write
\begin{equation}\label{eq:BrA}
B_r(A)
:=\{x\in M:\operatorname{dist}(x,A)<r\},
\qquad
\operatorname{dist}(x,A):=\inf_{y\in A} d(x,y).
\end{equation}
When the radius is fixed to $\delta>0$, we simply write $B_\delta(A)$.
\section{Approximate transitivity and mixing in metric spaces}

We first introduce the basic approximate notions used throughout this section.
For the reader, as a convenience, we recall the definition of $\delta$-topological transitivity from \cite{BenchihebAlshammari2024}.

\begin{definition}[$\delta$-topological transitivity ($\delta$-TT)]
\label{def:delta-TT}
Let $(M,d)$ be a metric space and let $f:M\to M$ be continuous.
Fix $\delta>0$.
We say that $f$ is \emph{$\delta$-topologically transitive} ($\delta$-TT)
if for every pair of nonempty open sets $U,V\subset M$,
there exists $n\in\mathbb{N}_0$ such that
\[
f^n(U)\cap B_\delta(V)\neq\varnothing.
\]
Equivalently, for every pair of nonempty open sets $U,V\subset M$,
there exist $x\in U$, $y\in V$, and $n\in\mathbb{N}_0$ such that
\[
d(f^n(x),y)<\delta.
\]
\end{definition}
We now introduce the notion of $\delta$-topological mixing in metric spaces.
\begin{definition}[$\delta$-topological mixing ($\delta$-TM)]
\label{def:delta-TM}
Let $(M,d)$ be a metric space and let $f:M\to M$ be a continuous map.
Fix $\delta>0$.
We say that $f$ is \emph{$\delta$-topologically mixing}
(\emph{$\delta$-TM}) if for every pair of nonempty open sets
$U,V\subset M$, there exists $N\in\mathbb{N}$ such that
\[
\inf_{x\in U,\,y\in V} d\bigl(f^n(x),y\bigr) < \delta,
\qquad \text{for all } n\ge N.
\]
Equivalently,
\[
f^n(U)\cap B_\delta(V)\neq\varnothing,
\qquad \text{for all } n\ge N.
\]
\end{definition}
We next introduce a stronger variant of $\delta$-topological mixing based on a fixed uniform radius strictly below $\delta$.
\begin{definition}[Uniform-from-below $\delta$-topological mixing (UFB-$\delta$-TM)]
\label{def:UFB-delta-TM}
Let $(M,d)$ be a metric space and let $f:M\to M$ be a continuous map.
Fix $\delta>0$.
We say that $f$ is \emph{uniform-from-below $\delta$-topologically mixing}
(\emph{UFB-$\delta$-TM}) if there exists $\eta\in(0,\delta)$ such that
for every pair of nonempty open sets $U,V\subset M$, there exists
$N\in\mathbb{N}$ such that
\[
f^n(U)\cap B_{\eta}(V)\neq\varnothing,
\qquad \text{for all } n\ge N.
\]
\end{definition}

The following notion is a metric-dynamical adaptation of the idea of a $\delta$-dense orbit.

\begin{definition}[$\delta$-transitive point]
\label{def:delta-transitive-point}
Let $(M,d)$ be a metric space and let $f:M\to M$ be a continuous map.
Fix $\delta>0$.
A point $x\in M$ is called \emph{$\delta$-transitive} for $f$ if for every
nonempty open set $V\subset M$,
\[
\inf_{n\ge0}\operatorname{dist}\bigl(f^n(x),V\bigr)<\delta.
\]
Equivalently, the forward orbit $\{f^n(x):n\ge0\}$ is $\delta$-dense in $M$.
\end{definition}


Now we compare the notions introduced above.  
We begin with two simple metric constructions that will be used repeatedly.

We begin with two elementary metric constructions that will be used repeatedly in the sequel.

\begin{lemma}[Bounded (capped) metric]
\label{lem:bounded-metric}
Let $(M,\rho)$ be a metric space and define
\(
d_b(x,y):=\min\{1,\rho(x,y)\} \),  \(x,y\in M.
\)
Then:
\begin{enumerate}
	\item $d_b$ is a metric compatible with the topology induced by $\rho$;
	\item $\operatorname{diam}(M,d_b)\le 1$;
	\item for any nonempty $V\subset M$ and $r>0$,
	\[
	r>1 \;\Rightarrow\; B^{d_b}_r(V)=M,
	\qquad
	0<r\le 1 \;\Rightarrow\; B^{d_b}_r(V)
	=\{x\in M:\operatorname{dist}_\rho(x,V)<r\},
	\]
	where $B^{d_b}_r(V)$ denotes the $r$-neighbourhood of $V$ with
	respect to $d_b$.
\end{enumerate}
\end{lemma}

\begin{proof}
The map $d_b(x,y):=\min\{1,\rho(x,y)\}$ is clearly symmetric, nonnegative, and vanishes exactly when $x=y$. For the triangle inequality, if $x,y,z\in M$, then
\(
\rho(x,z)\le \rho(x,y)+\rho(y,z),
\)
hence
\(
d_b(x,z)\le \min\{1,\rho(x,y)+\rho(y,z)\}.
\)
Using $\min\{1,a+b\}\le \min\{1,a\}+\min\{1,b\}$ for all $a,b\ge0$, we obtain
\(
d_b(x,z)\le d_b(x,y)+d_b(y,z).
\)
Thus $d_b$ is a metric.

If $0<r\le1$, then $B_{d_b}(x,r)=B_\rho(x,r)$, since
$d_b(x,y)<r$ is equivalent to $\rho(x,y)<r$. It follows immediately that $d_b$ and $\rho$ induce the same topology.

Also, $d_b(x,y)\le1$ for all $x,y\in M$, so $\operatorname{diam}(M,d_b)\le1$.

Finally, let $V\subset M$ be nonempty. If $r>1$, then $d_b(x,y)\le1<r$ for all $x\in M$ and $y\in V$, so $B_r^{d_b}(V)=M$.

If $0<r\le1$, then for every $x\in M$,
\[
x\in B_r^{d_b}(V)
\iff
\inf_{y\in V} d_b(x,y)<r
\iff
\inf_{y\in V} \rho(x,y)<r.
\]
Therefore
\(
B_r^{d_b}(V)=\{x\in M:\operatorname{dist}_\rho(x,V)<r\}.
\)
This completes the proof.
\end{proof}

\begin{lemma}[Neighbourhoods above the diameter]
\label{lem:diameter-neighborhood}
Let $(M,d)$ be a metric space with finite diameter
\(
D=\operatorname{diam}(M,d)<\infty.
\)
Then for any nonempty subset $V\subset M$ and any $\delta>D$, one has
\(
B_\delta(V)=M.
\)
Moreover, for $\delta=D$,
\(
B_D(V)=\{x\in M:\operatorname{dist}(x,V)<D\}.
\)
In particular, if $\operatorname{dist}(x,V)<D$ for every $x\in M$, then
\(
B_D(V)=M.
\)
\end{lemma}

\begin{proof}
Let $V\subset M$ be nonempty.

If $\delta>D$, then for every $x\in M$ and every $y\in V$,
\(
d(x,y)\le D<\delta.
\)
Hence
\(
\operatorname{dist}(x,V)\le D<\delta,
\)
so $x\in B_\delta(V)$. Therefore
\(
B_\delta(V)=M.
\)

For $\delta=D$, the identity
\(
B_D(V)=\{x\in M:\operatorname{dist}(x,V)<D\}
\)
is simply the definition of the open $D$-neighbourhood of $V$.
Thus, if $\operatorname{dist}(x,V)<D$ for every $x\in M$, then
\(
B_D(V)=M.
\)
\end{proof}

We now record the basic implications between the approximate mixing and transitivity notions introduced above.

\begin{proposition}
\label{prop:implications}
Let $(M,d)$ be a metric space, let $f:M\to M$ be continuous, and fix $\delta>0$.
Then
\[
\mathrm{UFB\mbox{-}}\delta\text{-TM}
\;\Longrightarrow\;
\delta\text{-TM}
\;\Longrightarrow\;
\delta\text{-TT}.
\]
\end{proposition}

\begin{proof}
Assume first that $f$ is UFB-$\delta$-TM. Then there exists $\eta\in(0,\delta)$ such that for every pair of nonempty open sets $U,V\subset M$, there exists $N\in\mathbb N$ such that $f^n(U)\cap B_\eta(V)\neq\varnothing$ for all $n\ge N$. Since $\eta<\delta$, we have $B_\eta(V)\subset B_\delta(V)$. Therefore $f^n(U)\cap B_\delta(V)\neq\varnothing$ for all $n\ge N$, and thus $f$ is $\delta$-TM.

Now assume that $f$ is $\delta$-TM. Let $U,V\subset M$ be nonempty open sets. Then there exists $N\in\mathbb N$ such that $f^n(U)\cap B_\delta(V)\neq\varnothing$ for all $n\ge N$. In particular, taking any $n\ge N$, we get $f^n(U)\cap B_\delta(V)\neq\varnothing$. Hence $f$ is $\delta$-TT.
\end{proof}

\begin{remark}[Comparison between classical and $\delta$-topological mixing]
The classical notion of topological mixing requires that, for every pair of nonempty
open sets $U,V\subset M$, there exists $N\in\mathbb{N}$ such that
\[
f^n(U)\cap V\neq\varnothing \quad \text{for all } n\ge N.
\]
Thus, the iterates of $U$ are eventually required to intersect the target open set
$V$ itself.

By contrast, $\delta$-topological mixing allows a spatial tolerance $\delta>0$:
for every pair of nonempty open sets $U,V\subset M$, there exists $N\in\mathbb{N}$
such that
\[
f^n(U)\cap B_\delta(V)\neq\varnothing \quad \text{for all } n\ge N,
\]
where $B_\delta(V)$ denotes the open $\delta$-neighbourhood of $V$.

Therefore, $\delta$-topological mixing is a quantitative weakening of classical
topological mixing. Moreover, if $0<\delta_1\le \delta_2$, then
$\delta_1$-TM implies $\delta_2$-TM, since $B_{\delta_1}(V)\subset B_{\delta_2}(V)$
for every nonempty set $V\subset M$.
\end{remark}

We now give an example showing that $\delta$-topological mixing is strictly weaker than classical topological mixing, and that the $\delta$-notions depend essentially on the underlying metric.

\begin{example}[Identity map under a bounded metric]
\label{ex:linear}
Let $(M,\rho)$ be a nontrivial normed space endowed with the capped metric
$ d_b(x,y)=\min\{1,\rho(x,y)\} $, introduced in Lemma~\ref{lem:bounded-metric}.
Then $\operatorname{diam}(M,d_b)\le 1$. Fix $\delta>1$ and consider the identity map
$f=\mathrm{Id}_M$.

We first show that $f$ is $\delta$-topologically mixing. Let $U,V\subset M$ be nonempty open sets. Since $\delta>\operatorname{diam}(M,d_b)$, Lemma~\ref{lem:diameter-neighborhood} yields $B_\delta(V)=M$. Hence, for every $n\in\mathbb N_0$, we have $f^n(U)=U$, so $f^n(U)\cap B_\delta(V)\neq\varnothing$. Thus $f$ is $\delta$-topologically mixing.

In fact, $f$ is even UFB-$\delta$-TM. Indeed, choose any $\eta\in(1,\delta)$. Since $\eta>\operatorname{diam}(M,d_b)$, Lemma~\ref{lem:diameter-neighborhood} again gives $B_\eta(V)=M$ for every nonempty subset $V\subset M$. It follows that for every pair of nonempty open sets $U,V\subset M$ and every $n\in\mathbb N_0$, we have $f^n(U)=U$, and hence $f^n(U)\cap B_\eta(V)\neq\varnothing$. Therefore $f$ is UFB-$\delta$-TM.

We now show that $f$ is not topologically mixing in the classical sense. Since $M$ is a nontrivial normed space, there exist distinct points $a,b\in M$. Let $r:=\rho(a,b)/4>0$, and define $U:=B_\rho(a,r)$ and $V:=B_\rho(b,r)$. Then $U$ and $V$ are nonempty open sets. Moreover, if $x\in U$ and $y\in V$, then
$\rho(x,y)\ge \rho(a,b)-\rho(x,a)-\rho(y,b)> \rho(a,b)-2r=\rho(a,b)/2>0$, so $U\cap V=\varnothing$.
Since $f^n(U)=U$ for all $n\in\mathbb N_0$, we obtain
\[
f^n(U)\cap V=\varnothing \qquad \text{for all } n\in\mathbb N_0.
\]
Thus $f$ is not topologically mixing.

Consequently, $\delta$-topological mixing does not imply classical topological mixing in general. This example also shows that the approximate notions may become trivial when the tolerance parameter exceeds the diameter of the underlying metric space.
\end{example}

The next example shows that the converse implication
\[
\delta\text{-TM}\Longrightarrow \mathrm{UFB\mbox{-}}\delta\text{-TM}
\]
fails even for a classical nonlinear minimal system.

\begin{example}[Irrational rotation: a nonlinear counterexample]
\label{ex:rotation}
Let \(M=\mathbb T^1=\mathbb R/\mathbb Z\) be the unit circle endowed with the
geodesic distance \(d\), so that \(\operatorname{diam}(M,d)=\tfrac12\).
Let \(f(x)=x+\alpha \pmod 1\), where \(\alpha\in\mathbb R\setminus\mathbb Q\), and fix
\(\delta=\tfrac12\).

We first show that \(f\) is \(\delta\)-topologically mixing.
Let \(V\subset M\) be a nonempty open set. Since \(V\) contains a nonempty open arc,
for every \(x\in M\) there exists \(y\in V\) with \(d(x,y)<\tfrac12\). Hence
\(B_{1/2}(V)=M\). Therefore, for any nonempty open sets \(U,V\subset M\) and any
\(n\in\mathbb N_0\), we have \(f^n(U)\subset M=B_{1/2}(V)\), so
\(f^n(U)\cap B_{1/2}(V)\neq\varnothing\). Thus \(f\) is \(\tfrac12\)-topologically
mixing.

We now show that \(f\) is not UFB-\(\delta\)-TM. Assume for contradiction that
\(f\) is UFB-\(\tfrac12\)-TM. Then there exists \(\eta\in(0,\tfrac12)\) such that
for every pair of nonempty open sets \(U,V\subset M\), there exists \(N\in\mathbb N\)
such that
\[
f^n(U)\cap B_\eta(V)\neq\varnothing
\qquad \text{for all } n\ge N.
\]

Choose a nonempty open arc \(V\subset M\) so small that \(B_\eta(V)\neq M\), and set
\(W:=B_\eta(V)\). Then \(W\) is a nonempty proper open subset of \(M\). By the above
property, for every nonempty open set \(U\subset M\), there exists \(N\in\mathbb N\)
such that
\[
f^n(U)\cap W\neq\varnothing
\qquad \text{for all } n\ge N.
\]
Thus \(f\) satisfies the classical topological mixing condition for the pair
\((U,W)\) and every nonempty open set \(U\subset M\).

However, irrational rotations on \(\mathbb T^1\) are not topologically mixing.
This contradiction shows that \(f\) is not UFB-\(\tfrac12\)-TM.

Consequently, \(\delta\)-topological mixing does not imply
UFB-\(\delta\)-topological mixing in general.
\end{example}

The next example shows that $\delta$-topological transitivity does not imply $\delta$-topological mixing, even in a classical minimal nonlinear system.

\begin{example}[Irrational rotation: $\delta$-TT does not imply $\delta$-TM]
\label{ex:irrational-deltaTT}
Let \(M=\mathbb T^1=\mathbb R/\mathbb Z\) be the unit circle endowed with the
geodesic distance \(d\), and let \(f(x)=x+\alpha \pmod 1\), where
\(\alpha\in\mathbb R\setminus\mathbb Q\). Fix \(0<\delta<\frac12\).

Then \(f\) is \(\delta\)-topologically transitive, but \(f\) is not
\(\delta\)-topologically mixing.

Indeed, since irrational rotations are minimal, they are in particular
topologically transitive. Hence, for every pair of nonempty open sets
\(U,V\subset M\), there exists \(n\in\mathbb N_0\) such that
\(f^n(U)\cap V\neq\varnothing\). Since \(V\subset B_\delta(V)\), it follows that
\(f^n(U)\cap B_\delta(V)\neq\varnothing\). Therefore \(f\) is
\(\delta\)-topologically transitive.

We now show that \(f\) is not \(\delta\)-topologically mixing. Choose nonempty open
arcs \(U,V\subset M\) so small that \(|U|+|V|+2\delta<1\), and set
\(W:=B_\delta(V)\). Then \(W\) is an open arc of length \(|W|=|V|+2\delta\), so
\(|U|+|W|<1\).

Fix representatives of the arcs on \(\mathbb R/\mathbb Z\), and write
\(W=(a,b)\) and \(U=(c,d)\), so that \(b-a=|W|\) and \(d-c=|U|\). Let
\[
A:=\{t\in\mathbb T^1:(U+t)\cap W\neq\varnothing\}.
\]
Indeed, if \(U\) is represented by an interval \((c,d)\), then \((U+t)\cap W\neq\varnothing\)
exactly when \(t\) belongs to the interval \((a-d,b-c)\) modulo \(1\), whose length is
\((b-a)+(d-c)=|W|+|U|<1\). Thus \(A\) is an open arc of length \(|W|+|U|<1\), and in particular
\(A\neq\mathbb T^1\).

For each \(n\ge0\), we have \(f^n(U)=U+n\alpha\). Hence
\[
f^n(U)\cap W\neq\varnothing
\quad\Longleftrightarrow\quad
n\alpha \pmod 1 \in A.
\]
Since \(\alpha\) is irrational, the orbit \(\{n\alpha \pmod 1:n\ge0\}\) is dense in
\(\mathbb T^1\). As \(\mathbb T^1\setminus A\) contains a nonempty open arc, there
exist infinitely many \(n\in\mathbb N\) such that \(n\alpha \pmod 1\notin A\). For
those \(n\), we have \(f^n(U)\cap W=\varnothing\), that is,
\(f^n(U)\cap B_\delta(V)=\varnothing\).

Thus the eventual-intersection condition in the definition of
\(\delta\)-topological mixing fails. Therefore \(f\) is not
\(\delta\)-topologically mixing.

Consequently, for every \(0<\delta<\frac12\), \(\delta\)-topological transitivity
does not imply \(\delta\)-topological mixing.
\end{example}

\begin{remark}[Topological versus metric mixing]
The previous examples highlight a distinction between
classical topological mixing and $\delta$-topological mixing.
Classical topological mixing is formulated purely in terms of intersections
of open sets, and therefore depends only on the topology of the underlying space.
By contrast, $\delta$-topological mixing depends explicitly on the metric,
since it requires eventual intersections with the $\delta$-neighbourhood
\(B_\delta(V)\) of a target set \(V\).

Thus, unlike classical topological mixing, the notion of $\delta$-topological
mixing is sensitive to the choice of metric. In this sense, it should be viewed
as a metric version of eventual mixing up to a prescribed spatial tolerance.
\end{remark}
\section{$\delta$-Hypercyclicity Criterion}

We begin by recalling a standard formulation of the classical Hypercyclicity Criterion.

\begin{definition}[Classical Hypercyclicity Criterion]
\label{def:HC-classical}
Let $X$ be a separable $F$-space and let $T:X\to X$ be a continuous linear operator.
We say that $T$ satisfies the \emph{Hypercyclicity Criterion} if there exist a strictly
increasing sequence of integers $(n_k)$, dense subsets $V,W\subset X$, and maps
$S_{n_k}:W\to X$ such that:
\begin{enumerate}
	\item for every $v\in V$, $T^{n_k}v\to 0$;
	\item for every $w\in W$, $S_{n_k}w\to 0$;
	\item for every $w\in W$, $T^{n_k}(S_{n_k}w)\to w$.
\end{enumerate}
\end{definition}

It is classical that, on separable $F$-spaces, the Hypercyclicity Criterion implies hypercyclicity; see, for example, \cite{GEP2011}.

We now formulate a $\delta$-approximate version of the Hypercyclicity Criterion in the metric setting.

\begin{definition}[$\delta$-Hypercyclicity Criterion]
\label{def:delta-HC-criterion}
Let $X$ be a separable topological vector space equipped with a translation-invariant
metric $d$ generating its topology, and let $T:X\to X$ be a continuous linear operator.
Fix $\delta>0$.

We say that $T$ satisfies the \emph{$\delta$-Hypercyclicity Criterion} if there exist a
strictly increasing sequence of integers $(n_k)$, dense subsets $V,W\subset X$, and maps
$S_{n_k}:W\to X$ such that:
\begin{enumerate}
	\item for every $v\in V$, $T^{n_k}v\to 0$;
	\item for every $w\in W$, $S_{n_k}w\to 0$;
	\item for every $w\in W$, there exists $k_0\in\mathbb N$ such that
	$d\bigl(T^{n_k}(S_{n_k}w),w\bigr)<\delta$ for all $k\ge k_0$.
\end{enumerate}
\end{definition}
We now show that the $\delta$-Hypercyclicity Criterion yields
$\delta$-hypercyclicity in the setting of separable F-spaces.

\begin{proposition}
\label{prop:criterion-implies-delta-hypercyclic}
Let $X$ be a separable F-space equipped with a translation-invariant metric $d$,
and let $T:X\to X$ be a continuous linear operator.
If $T$ satisfies the $\delta$-Hypercyclicity Criterion, then $T$ is
$\delta$-hypercyclic.
\end{proposition}

\begin{proof}
Let $V,W\subset X$ be the dense subsets appearing in
Definition~\ref{def:delta-HC-criterion}.

We first show that $T$ is $\delta$-topologically transitive. Let $U,O\subset X$
be nonempty open sets, and choose $v\in V\cap U$ and $w\in W\cap O$.

Since $U-v$ is an open neighbourhood of $0$ and $S_{n_k}w\to0$, there exists
$k_1$ such that $v+S_{n_k}w\in U$ for all $k\ge k_1$.
Since $O-w$ is an open neighbourhood of $0$ and $T^{n_k}v\to0$, there exists
$k_2$ such that $w+T^{n_k}v\in O$ for all $k\ge k_2$.
By condition \textup{(3)}, there exists $k_3$ such that
$d\bigl(T^{n_k}(S_{n_k}w),w\bigr)<\delta$ for all $k\ge k_3$.

Fix $k\ge \max\{k_1,k_2,k_3\}$ and set $x:=v+S_{n_k}w$. Then $x\in U$, and
\[
d\bigl(T^{n_k}x,\;w+T^{n_k}v\bigr)
=
d\bigl(T^{n_k}(S_{n_k}w),w\bigr)
<\delta.
\]
Since $w+T^{n_k}v\in O$, it follows that $T^{n_k}x\in B_\delta(O)$. Hence
$T^{n_k}(U)\cap B_\delta(O)\neq\varnothing$, and therefore $T$ is
$\delta$-topologically transitive.

Let $(O_m)_{m\ge1}$ be a countable base of nonempty open sets of $X$, which exists
because $X$ is separable and metrizable. For each $m\ge1$, define
\[
E_m:=\bigcup_{n\ge0} T^{-n}\bigl(B_\delta(O_m)\bigr).
\]
Each $E_m$ is open, since $T$ is continuous and $B_\delta(O_m)$ is open.

We claim that each $E_m$ is dense. Let $U\subset X$ be a nonempty open set. Since
$T$ is $\delta$-topologically transitive, there exists $n\ge0$ such that
$T^n(U)\cap B_\delta(O_m)\neq\varnothing$. Hence there exists $x\in U$ with
$T^n x\in B_\delta(O_m)$, that is, $x\in U\cap E_m$. Thus $U\cap E_m\neq\varnothing$,
and $E_m$ is dense.

Since $X$ is an F-space, it is a Baire space. Therefore the intersection
\[
E:=\bigcap_{m\ge1} E_m
\]
is dense and in particular nonempty. Let $x\in E$. Then for every $m\ge1$, there
exists $n\ge0$ such that $T^n x\in B_\delta(O_m)$.

Now let $O\subset X$ be any nonempty open set. Choose a basis element $O_m$ such that
$O_m\subset O$. Since $x\in E_m$, there exists $n\ge0$ with $T^n x\in B_\delta(O_m)$.
Because $O_m\subset O$, we have $B_\delta(O_m)\subset B_\delta(O)$, and hence
$T^n x\in B_\delta(O)$. This shows that the orbit of $x$ is $\delta$-dense in $X$.

Therefore $x$ is a $\delta$-hypercyclic vector for $T$, and so $T$ is
$\delta$-hypercyclic.
\end{proof}

We next show that the classical Hypercyclicity Criterion implies its approximate $\delta$-version.

\begin{proposition}[Classical HC $\Rightarrow$ $\delta$-HC]
\label{prop:HC-implies-deltaHC}
Let $X$ be a separable F-space equipped with a translation-invariant metric $d$
compatible with its topology, and let $T:X\to X$ be a continuous linear operator.

If $T$ satisfies the classical Hypercyclicity Criterion, then it satisfies
the $\delta$-Hypercyclicity Criterion for every $\delta>0$.
\end{proposition}

\begin{proof}
If $T$ satisfies the classical Hypercyclicity Criterion, then there exist dense
sets $V,W\subset X$, a strictly increasing sequence $(n_k)$, and maps
$S_{n_k}:W\to X$ such that, for all $v\in V$ and $w\in W$, one has
$T^{n_k}v\to 0$, $S_{n_k}w\to 0$, and $T^{n_k}(S_{n_k}w)\to w$.

Since $d$ is compatible with the topology of $X$, the convergence
$T^{n_k}(S_{n_k}w)\to w$ implies
$d\bigl(T^{n_k}(S_{n_k}w),w\bigr)\to 0$.
Therefore, for every fixed $\delta>0$ and every $w\in W$, there exists
$k_0\in\mathbb N$ such that
$d\bigl(T^{n_k}(S_{n_k}w),w\bigr)<\delta$ for all $k\ge k_0$.
Thus condition \textup{(3)} of the $\delta$-Hypercyclicity Criterion holds.
Conditions \textup{(1)} and \textup{(2)} are identical in the two criteria.
Hence $T$ satisfies the $\delta$-Hypercyclicity Criterion.
\end{proof}

We conclude this section by observing that the $\delta$-Hypercyclicity Criterion yields eventual $\delta$-mixing along the prescribed sequence $(n_k)$.

\begin{proposition}[Mixing along the prescribed sequence]
\label{prop:delta-HC-sequence}
Let $X$ be a separable F-space equipped with a translation-invariant metric
$d$ compatible with its topology, and let $T:X\to X$ be a continuous linear
operator. Fix $\delta>0$.
Assume that $T$ satisfies the $\delta$-Hypercyclicity Criterion along a
strictly increasing sequence $(n_k)$.
Then for every pair of nonempty open sets $U,O\subset X$, there exists
$K\in\mathbb N$ such that
\[
T^{n_k}(U)\cap B_\delta(O)\neq\varnothing
\qquad \text{for all } k\ge K.
\]
\end{proposition}

\begin{proof}
Let $V,W\subset X$ be the dense subsets appearing in
Definition~\ref{def:delta-HC-criterion}.

Let $U,O\subset X$ be nonempty open sets, and choose
$v\in V\cap U$ and $w\in W\cap O$.

Since $U-v$ is an open neighbourhood of $0$ and $S_{n_k}w\to0$,
there exists $K_1\in\mathbb N$ such that $v+S_{n_k}w\in U$ for all $k\ge K_1$.
Since $O-w$ is an open neighbourhood of $0$ and $T^{n_k}v\to0$,
there exists $K_2\in\mathbb N$ such that $w+T^{n_k}v\in O$ for all $k\ge K_2$.
By condition \textup{(3)} of the $\delta$-Hypercyclicity Criterion, there exists
$K_3\in\mathbb N$ such that
$d\bigl(T^{n_k}(S_{n_k}w),w\bigr)<\delta$ for all $k\ge K_3$.

Let $K:=\max\{K_1,K_2,K_3\}$. Fix $k\ge K$, and define $x_k:=v+S_{n_k}w$.
Then $x_k\in U$, and
\[
d\bigl(T^{n_k}x_k,\;w+T^{n_k}v\bigr)
=
d\bigl(T^{n_k}(S_{n_k}w),w\bigr)
<\delta.
\]
Since $w+T^{n_k}v\in O$, it follows that $T^{n_k}x_k\in B_\delta(O)$.
Hence
\[
T^{n_k}(U)\cap B_\delta(O)\neq\varnothing
\qquad \text{for all } k\ge K.
\]
This proves the result.
\end{proof}

\section{Weighted shifts}

Weighted backward shifts are among the most classical examples in linear
dynamics. Their action is explicit, and they play a central role in the
study of hypercyclicity, mixing, and related dynamical properties.

A foundational result in this direction is due to Rolewicz
\cite{Rolewicz1969}, who proved that the operator $\lambda B$ is
hypercyclic on $\ell^p$ whenever $|\lambda|>1$.
Weighted shifts also provide a standard source of examples and
counterexamples in the modern theory of linear dynamics. In particular,
Salas obtained a characterization of hypercyclic weighted shifts; see
\cite{Salas1995}. For broader background on linear dynamics and weighted
shifts, we also refer to \cite{GEP2011}.

We now recall the corresponding definitions.

\begin{definition}[Weighted backward shifts]
Let $X=\ell^{2}(\mathbb N_0)$ in the unilateral case, or
$X=\ell^{2}(\mathbb Z)$ in the bilateral case.
Let $(e_i)$ denote the canonical basis, and let $(a_i)$ be a bounded sequence
of nonzero scalars.

The \emph{weighted backward shift} $T$ is defined by
\[
T(e_i)=a_i e_{i-1},
\]
for $i\ge 1$ in the unilateral case, and for $i\in\mathbb Z$ in the bilateral
case, with the convention $T(e_0)=0$ in the unilateral case.

Its iterates satisfy, in the unilateral case,
\[
T^n e_k=
\begin{cases}
	\left(\displaystyle\prod_{j=k-n+1}^{k} a_j\right)e_{k-n}, & k\ge n,\\[0.2cm]
	0, & k<n,
\end{cases}
\]
and, in the bilateral case,
\[
T^n e_k=
\left(\displaystyle\prod_{j=k-n+1}^{k} a_j\right)e_{k-n},
\qquad k\in\mathbb Z,\ n\ge 0.
\]

For later use, we denote the corresponding weight products by
\[
A_k^{(n)}:=\prod_{j=k-n+1}^{k} a_j,
\]
whenever the indices are in the admissible range.
\end{definition}

In the classical setting, the following mixing characterisation for weighted
backward shifts is standard; see, for example, \cite{GEP2011}.

\begin{theorem}[Classical mixing theorem for weighted shifts]
\label{thm:classical-mixing-shifts}
Let $T$ be a weighted backward shift.

\begin{enumerate}
	\item[\textnormal{(a)}]
	In the unilateral case, $T$ is topologically mixing if and only if
	$\lim_{n\to\infty}\prod_{i=1}^n |a_i|=\infty$.
	
	\item[\textnormal{(b)}]
	In the bilateral case, $T$ is topologically mixing if and only if
	$\lim_{n\to\infty}\prod_{i=1}^n |a_i|=\infty$ and
	$\lim_{n\to\infty}\prod_{i=0}^n |a_{-i}|=\infty$.
\end{enumerate}
\end{theorem}

In the remainder of this section we discuss approximate analogues of these
classical conditions.

\begin{proposition}
Let $T$ be a unilateral weighted backward shift on $\ell^2(\mathbb N_0)$
with weight sequence $(a_i)_{i\ge1}$, and let $\delta>0$.
If $\lim_{n\to\infty}\prod_{i=1}^n |a_i|=\infty$, then $T$ is
$\delta$-topologically mixing.
\end{proposition}

\begin{proof}
Under the above condition, $T$ is topologically mixing by the classical
mixing theorem for unilateral weighted shifts. Hence, for every pair of
nonempty open sets $U,V$, there exists $N$ such that
$T^n(U)\cap V\neq\varnothing$ for all $n\ge N$.
Since $V\subset B_\delta(V)$, it follows that
$T^n(U)\cap B_\delta(V)\neq\varnothing$ for all $n\ge N$.
Therefore $T$ is $\delta$-topologically mixing.
\end{proof}
\begin{remark}
\label{rem:delta-hypercyclic-not-delta-mixing}
There exist bounded sequences of strictly positive weights $(a_i)_{i\ge1}$ such that
\[
\sup_{n\ge1}\prod_{i=1}^{n} a_i=\infty,
\qquad\text{but}\qquad
\lim_{n\to\infty}\prod_{i=1}^{n} a_i\neq\infty.
\]
For such a sequence, the associated unilateral weighted backward shift may be
hypercyclic, and hence $\delta$-hypercyclic for every $\delta>0$.

On the other hand, the classical mixing criterion for unilateral weighted shifts
requires
\[
\lim_{n\to\infty}\prod_{i=1}^{n} |a_i|=\infty.
\]
Thus there exist unilateral weighted shifts that are hypercyclic, but not
topologically mixing. This shows that, even in the weighted-shift setting,
hypercyclicity is strictly weaker than classical topological mixing.
\end{remark}


The classical bilateral mixing criterion immediately yields a sufficient condition for $\delta$-topological mixing.
We now consider bilateral weighted backward shifts on $\ell^2(\mathbb Z)$
endowed with its usual norm.

\begin{proposition}
\label{prop:delta-mixing-bilateral-safe}
Let $T$ be a bilateral weighted backward shift on $\ell^2(\mathbb Z)$
with weight sequence $(a_i)_{i\in\mathbb Z}$, and let $\delta>0$.
If $\lim_{n\to\infty}\prod_{i=1}^n |a_i|=\infty$ and
$\lim_{n\to\infty}\prod_{i=0}^n |a_{-i}|=\infty$, then $T$ is
$\delta$-topologically mixing.
\end{proposition}

\begin{proof}
Under the above assumptions, $T$ is topologically mixing by the classical
mixing theorem for bilateral weighted backward shifts. Hence, for every pair
of nonempty open sets $U,V\subset \ell^2(\mathbb Z)$, there exists $N\in\mathbb N$
such that $T^n(U)\cap V\neq\varnothing$ for all $n\ge N$. Since
$V\subset B_\delta(V)$, it follows that
$T^n(U)\cap B_\delta(V)\neq\varnothing$ for all $n\ge N$.
Therefore $T$ is $\delta$-topologically mixing.
\end{proof}

Combining the classical mixing criteria for unilateral and bilateral weighted
backward shifts with the definition of $\delta$-topological mixing, we obtain
the following sufficient condition.

\begin{corollary}
Let $T$ be a weighted backward shift.

\begin{enumerate}
	\item[\textnormal{(a)}]
	In the unilateral case, if $\lim_{n\to\infty}\prod_{i=1}^n |a_i|=\infty$,
	then $T$ is $\delta$-topologically mixing for every $\delta>0$.
	
	\item[\textnormal{(b)}]
	In the bilateral case, if $\lim_{n\to\infty}\prod_{i=1}^n |a_i|=\infty$ and
	$\lim_{n\to\infty}\prod_{i=0}^n |a_{-i}|=\infty$, then $T$ is
	$\delta$-topologically mixing for every $\delta>0$.
\end{enumerate}
\end{corollary}
\begin{proof}
Part \textup{(a)} follows from the unilateral case, and part \textup{(b)} from
the bilateral case.
\end{proof}
We conclude this section with the classical example of the weighted shift $\lambda B$.

\begin{proposition}[$\delta$-Hypercyclicity Criterion for $\lambda B$]
\label{prop:lambdaB-delta-HC}
Let $B$ be the backward shift on $\ell^2(\mathbb N_0)$, defined by
\[
B(x_0,x_1,x_2,\dots)=(x_1,x_2,x_3,\dots),
\]
and let $\lambda\in\mathbb C$ with $|\lambda|>1$.
Set $T:=\lambda B$, and let $d$ be the metric induced by the usual
Hilbert norm on $\ell^2(\mathbb N_0)$.

Then for every $\delta>0$, the operator $T$ satisfies the
$\delta$-Hypercyclicity Criterion with respect to the sequence $n_k=k$.
In particular, $T$ is $\delta$-hypercyclic for every $\delta>0$.
\end{proposition}

\begin{proof}
Fix $\delta>0$. Let $D\subset \ell^2(\mathbb N_0)$ be the set of finitely
supported sequences. Then $D$ is dense in $\ell^2(\mathbb N_0)$. Set $V=W=D$.

Let $v\in V$. Since $v$ is finitely supported, we have $B^n v=0$ for all
sufficiently large $n$, and therefore $T^n v=\lambda^n B^n v=0$ for all
sufficiently large $n$. Thus condition \textup{(1)} of
Definition~\ref{def:delta-HC-criterion} holds.

Now let $w=(w_0,w_1,w_2,\dots)\in W$, and define
\[
S_n(w):=(\underbrace{0,\dots,0}_{n\text{ terms}},\lambda^{-n}w_0,\lambda^{-n}w_1,\lambda^{-n}w_2,\dots).
\]
Then $\|S_n(w)\|_2=|\lambda|^{-n}\|w\|_2\to 0$ as $n\to\infty$, so condition
\textup{(2)} holds.

Finally, $T^nS_n(w)=(\lambda B)^nS_n(w)=\lambda^nB^nS_n(w)=w$. Hence
$d\bigl(T^nS_n(w),w\bigr)=0<\delta$ for all $n\in\mathbb N$, and condition
\textup{(3)} is satisfied.

Therefore $T$ satisfies the $\delta$-Hypercyclicity Criterion with respect to
the sequence $n_k=k$. The final assertion follows from
Proposition~\ref{prop:criterion-implies-delta-hypercyclic}.
\end{proof}

\section*{Funding}
This research received no external funding.

\section*{Conflict of Interest}
The authors declare that they have no conflict of interest regarding the publication of this paper.

\section*{Data Availability}
No data were generated or analyzed during this study.

\section*{Author Contributions}
All authors contributed equally to the conceptualization, writing, and revision of the manuscript, 
and approved the final version.



\begin{thebibliography}{99}
\bibitem{BayartMatheron2009}
Bayart, F., Matheron, \'E. (2009).
\newblock \textit{Dynamics of Linear Operators}.
\newblock Cambridge University Press.

\bibitem{BenchihebAlshammari2024}
Alshammari, H., Benchiheb, O. (2024).
\newblock Approximate topological transitivity in metric spaces.
\newblock \textit{Preprint}.

\bibitem{BesPeris1999}
B\`es, J., Peris, A. (1999).
\newblock Hereditarily hypercyclic operators.
\newblock \textit{Journal of Functional Analysis}, 167(1), 94--112.

\bibitem{Birkhoff1929}
Birkhoff, G. D. (1929).
\newblock D\'emonstration d'un th\'eor\`eme \'el\'ementaire sur les fonctions enti\`eres.
\newblock \textit{C. R. Acad. Sci. Paris}, 189, 473--475.

\bibitem{Feldman2002}
Feldman, N. S. (2002).
\newblock Perturbations of hypercyclic vectors.
\newblock \textit{Journal of Mathematical Analysis and Applications}, 273(1), 67--74.

\bibitem{GEP2011}
Grosse-Erdmann, K.-G., Peris, A. (2011).
\newblock \textit{Linear Chaos}.
\newblock Universitext, Springer, London.

\bibitem{GethnerShapiro1987}
Gethner, R. M., Shapiro, J. H. (1987).
\newblock Universal vectors for operators on spaces of holomorphic functions.
\newblock \textit{Proceedings of the American Mathematical Society}, 100(2), 281--288.

\bibitem{Kitai1982}
Kitai, C. (1982).
\newblock Invariant closed sets for linear operators.
\newblock Ph.D. Thesis, University of Toronto.

\bibitem{Rolewicz1969}
Rolewicz, S. (1969).
\newblock On orbits of elements.
\newblock \textit{Studia Mathematica}, 32, 17--22.

\bibitem{Salas1995}
Salas, H. N. (1995).
\newblock Hypercyclic weighted shifts.
\newblock \textit{Transactions of the American Mathematical Society}, 347(3), 993--1004.

\end{thebibliography}
\end{document}